\documentclass[12pt]{amsart}
\usepackage[all]{xy}   
\usepackage {amssymb,latexsym,amsthm,amsmath}
\usepackage{xcolor}
\usepackage{amsthm}
\usepackage{mathtools}
\usepackage[raiselinks=false,colorlinks=true,citecolor=blue,urlcolor=blue,
linkcolor=blue,bookmarksopen=true,pdftex]{hyperref}

\newcommand{\Aut}{\textup{Aut}}
\newcommand{\Out}{\textrm{Out}}
\newcommand{\trace}{\textrm{tr}}
\newcommand{\fix}{\textrm{Fix}}
\newcommand{\SL}{\textrm{SL}}
\newcommand{\GL}{\textrm{GL}}

\makeatletter
\def\imod#1{\allowbreak\mkern10mu({\operator@font mod}\,\,#1)}
\makeatother

\newtheorem{theorem}{Theorem}[section]
\newtheorem{lemma}[theorem]{Lemma}
\newtheorem{corollary}[theorem]{Corollary}

\newtheorem*{theorem*}{Theorem}
\theoremstyle{definition}
\newtheorem{remark}[theorem]{Remark}

\numberwithin{equation}{section}

\newcommand{\ignore}[1]{}

\newcommand{\mynote}[1]{}
\begin{document}
\setcounter{section}{0}
\title{Twisted conjugacy in classical groups over certain domains of Characteristic $p>0$}
\author{Shripad M. Garge}
\address{SMG: Department of Mathematics, Indian Institute of Technology, Powai, Mumbai, 400 076 India.}
\email{shripad@math.iitb.ac.in, smgarge@gmail.com}
\author{Oorna Mitra.}
\address{OM: CMI, Chennai} 
\email{oornam@cmi.ac.in, urna.mitra@gmail.com}


\begin{abstract}
Let $F$ be a subfield of the algebraic closure of a finite field $\mathbb{F}_p$, $p \ne 2$, and let $R$ denote any ring such that $F[t] \subset R \subsetneq F(t)$. 
Let $G$ be a classical Chevalley group of adjoint type defined over $R$. 
We prove that the group $G(R)$ has the $R_{\infty}$-property. 
\end{abstract}

\maketitle
\section{Introduction}

Let $G$ be a group and $\phi : G \to G$ be an automorphism. 
Then there is an action of $G$ on itself given by $g.x=gx\phi(g^{-1})$. 
This action is called the $\phi$-twisted conjugation action of $G$. 
The orbits of this action are called the $\phi$-twisted conjugacy classes. 
The number of the orbits is called the Reidemeister number of $\phi$ and is denoted by $R(\phi)$. 
If $R(\phi)$ is infinite for every automorphism $\phi$ of $G$, then $G$ is said to have the $R_\infty$-property and is called an $R_\infty$-group.

The notion of twisted conjugacy of groups first appeared in Nielsen fixed point theory. 
The question of which classes of groups possess $R_\infty$-property was formulated by Fel'shtyn and Hill in 1994 \cite{felshtyn-hill}. 
This and other related questions have been an active area of research for past few decades, of which one of the first general results was obtained by A. Fel'shtyn, G. Levitt and M. Lustig; they proved that non-elementary Gromov hyperbolic groups possess the $R_\infty$-property \cite{Fel, LL}. 
Since then many classes of groups have been classified according to whether or not they have the $R_\infty$-property.

In this paper, we address this question for classical Chevalley groups of adjoint type over rings that contain the polynomial algebra $F[t]$ and is properly contained in the rational function field $F(t)$, where $F$ is a subfield of the algebraic closure of $\mathbb{F}_p$.
Thus we build upon and generalise the results in \cite{MP} by Mitra and Sankaran, where the $R_\infty$-property of the general and special linear groups over the same algebras has been shown. 

Various other authors have studied $R_\infty$-property of linear groups. 
See \cite{FN}, \cite{Y}, \cite{ms-cmb}, \cite{ms-tg}, \cite{T}, \cite{N}, \cite{T1}. Nasybullov \cite{N} and Felshtyn-Nasybullov \cite{FN} proved that that a Chevally group of classical type over an algebraically closed field F of characteristic zero has the $R_\infty$-property
if and only if $F$ has finite transcendence degree over $\mathbb{Q}$. 
It follows from a classical theorem by Steinberg \cite{steinberg} that any connected linear algebraic group over an algebraically closed field of characteristic $p>0$ does not have the $R_\infty$-property. 
In contrast to this, our result shows the existence of linear algebraic groups over integral domains of characteristic $p>0$ having the $R_\infty$-property.

Throughout this paper, we shall denote by $R$ any ring with $F[t]\subset R \subsetneq F(t)$, where $F[t]$ is the polynomial algebra and $F(t)$ is the rational function field in one variable $t$ where $F$ is a subfield of $\bar{\mathbb F}_p$, the algebraic closure of the field $\mathbb F_p$, $p$ a prime.
Let $q=p^e, e\ge 1$. The main result of this paper is the following.

\begin{theorem} \label{MainThm}
Let $F$ be a subfield of $\bar{\mathbb F}_p$ with $p \neq 2$ and let $R$ be a ring such that $F[t] \subseteq R \subsetneq F(t)$.
Let $G$ be a simple classical Chevalley group of adjoint type of rank $>1$ defined over $R$. 

If $\mathsf{G}$ is a group such that $\mathsf{G}/\mathsf{N} \stackrel{\sim}{\rightarrow} G(R)$ for a characteristic subgroup $\mathsf{N}$ of $\mathsf{G}$ then $\mathsf{G}$ has the $R_\infty$-property.
\end{theorem}

\begin{remark} Since there are exactly $4$ types of classical Chevalley adjoint groups, $\mathsf{A}$--$\mathsf{D}$, we have the following $4$ cases:

\noindent {\sf (A)} $\mathsf{G}/\mathsf{N} \stackrel{\sim}{\rightarrow} PSL_n(R)$ for $n \geq 3$, this includes the case that $\mathsf{G}$ is $SL_n(R)$.
The $R_\infty$-property of $SL_n(R)$ has already been proved by Mitra and Sankaran in \cite{MP}.

\noindent {\sf (B)} $\mathsf{G}/\mathsf{N} \stackrel{\sim}{\rightarrow} PSO_{2n+1}(R)=SO_{2n+1}(R)$ for $n \geq 2$.

\noindent {\sf (C)} $\mathsf{G}/\mathsf{N} \stackrel{\sim}{\rightarrow} PSp_{2n}(R)$ for $n \geq 2$, this includes the case that $\mathsf{G} = Sp_{2n}(R)$.

\noindent {\sf (D)} $\mathsf{G}/\mathsf{N} \stackrel{\sim}{\rightarrow} PSO_{2n}(R)$ for $n \geq 3$, this includes the case that $\mathsf{G} = SO_{2n}(R)$.
\end{remark}

The proofs of the above four cases are given in \S\ref{A_n}, \S\ref{B_n} \S\ref{C_n},and \S\ref{D_n} respectively. 
The proof of the $R_\infty$-propery of the group $SO_8$ needs special treatment and is given in \S\ref{D_4}.
By Lemma \ref{characteristic} (1), we need to prove the results only for the corresponding adjoint groups. 

Roughly speaking, the proof of \ref{MainThm} involves three major steps which are as follows: \\
(i) to have a nice description of the automorphisms of the groups of our consideration, which is given by E. Bunina \cite{Bu} and is recalled in \S\ref{Aut} of this paper, \\
(ii) find a convenient representative automorphism $\phi$ within its outer automorphism class and use Lemma \ref{outer}, \\
(iii) come up with an infinite sequence $x_k, k \in \mathbb{N}$ of elements in the group which lie in distinct $\phi$-twisted conjugacy classes and thus conclude the $R_\infty$-property of the group.

\begin{remark}
We note that some low rank groups are left out from the Theorem \ref{MainThm}.
The group $PSL_2(R)$ is a difficult group to handle, in the sense that we do not have a complete description of its automorphisms.
The groups $PSO_3(R)$ and $PSp_2(R)$ are isomorphic to $PSL_2(R)$ while the group $PSO_4(R)$ is isomorphic to $(PSL_2 \times PSL_2)(R)$. 
These isomorphisms of Chevalley groups can be read off from the corresponding isomorphisms in Lie algebras which are clear by looking at the corresponding Dynkin diagrams. The $R_\infty$-property of some rank 1 groups has been established in \cite{mitra-sankaran-n2} where it is proved that the groups $SL_2(F[t]), GL_2(F[t]), GL_2(\mathbb{F}_q[t,t^{-1}])$ have the $R_\infty$-property where $F$ is a subfield of $\bar{\mathbb{F}_p}$.
\end{remark}

\section{Basic results on twisted conjugacy}

Let $\phi:G\to G$ be an automorphism of a
group $G$.
We shall denote by $[x]_\phi$ the $\phi$-twisted conjugacy class of $x$ and by $\mathcal R(\phi)$ the set of all $\phi$-twisted conjugacy classes in $G$.
In this section, we present some basic results concerning twisted conjugacy and the $R_\infty$-property which are relevant to our purpose. 

The first step towards simplifying the $R_\infty$-problem is to notice the fact that there is a bijection from $\mathcal R(\iota_g\circ \phi)$ to $\mathcal R(\phi)$ defined as $[x]_{\iota_g\circ\phi}\mapsto [xg]_{\phi}$, where $\iota_g:G\to G$ denotes the inner automorphism $x\mapsto gxg^{-1}$. 
We have the following lemma which is well-known and can be found for example, in \cite[\S3]{g-s}.
\begin{lemma}\label{outer}
The group $G$ has the $R_\infty$-property if $R(\phi)=\infty$ for all $\phi\in T$ where $T \subset Aut(G)$ is a set of representatives of $\Out(G)$, the outer automorphism group of $G$. 
\end{lemma} 

The following lemma gives a way of reducing the twisted conjugacy problem to usual conjugacy problem in case of finite order automorphisms. 
The lemma can be proved along the same line as \cite[ Lemma 2.3]{gs-pjm-2016}. 

\begin{lemma}\label{fixedgroup}
Let $\theta:G\to G$ be an automorphism such that for every $z \in G, \theta^{n_z}(z)=z,$ for some natural number $n_z$. Suppose that $x,y \in \fix(\theta)$ and that $[x]_\theta=[y]_\theta$.
Then there exists $r \in \mathbb{N}$ such that $x^r$ and $y^r$ are conjugates in $G$. (This $r$ depends on $x$ and $y$.)

In particular, if $\theta$ is a finite order automorphism of order $r$, then $x,y \in \fix(\theta)$ with $[x]_\theta=[y]_\theta$ implies that $x^r$ and $y^r$ are conjugates in $G$.
\end{lemma} 

Let $K\subset G$ be a normal subgroup.
Let $\eta:G\to H$ be the canonical quotient map where $H=G/K$.
Suppose that $\phi:G\to G$ is an automorphism such that $\phi(K)=K$.
We then have the following diagram with exact rows where the isomorphisms $\phi',\bar\phi$ are defined by $\phi$ in a natural way:
$$\begin{array}{ l l l l l l llll }
1 &\to & K&\to & G& \to & H& \to &1\\
& &\downarrow\phi' && \downarrow \phi& &\downarrow \bar \phi &&\\
1&\to &K&\to &G& \to &H&\to& 1\\
\end{array}
$$
Then we have the following lemma from \cite[Lemma 2.2]{ms-cmb} that connects the Reidemeister numbers of $\phi, \phi'$ and $\bar{\phi}$.
\begin{lemma} \label{characteristic}
Suppose that $\phi:G \to G$ is an automorphism of an infinite group such that the rows in the above commutative diagram are exact and $\phi',\bar \phi$ are the induced isomorphisms. 
Then:
\begin{enumerate}
\item $R(\phi) \geq R(\bar\phi)$.

\item Suppose that $K$ is finite. 
Then $R(\bar\phi)=\infty$ if $R(\phi)=\infty$.  

\item Suppose that $H$ is finite. 
Then $R(\phi)=\infty$ if $R(\phi')=\infty$.  

\end{enumerate}
\end{lemma}

\section{Automorphisms of classical Chevalley groups} \label{Aut}

In this section, we describe all group automorphisms of classical Chevalley groups over commutative rings with identity. 

Let $\Phi$ be a root system, in the sense of Humphreys, (\cite[$\S 9.2$]{Hu}), and let $R$ be a commutative ring with identity. 
Let $G(\Phi, R)$ be the Chevalley group of adjoint type associated to $\Phi$ over $R$. 
For the basic theory and constructions of Chevalley groups we refer the reader to the excellent book by Steinberg, \cite{St}.

We assume that the root system $\Phi$ is irreducible of classical type, that is, we assume that $\Phi$ is $A_l$ for $l \geq 1$, $B_l$ for $l \geq 2$, $C_l$ for $l \geq 3$ or $D_l$ for $l \geq 4$. 
The automorphisms of $G(\Phi, R)$ have been discussed in various papers but our reference is \cite{Bu} since it is one of the recent works on this topic and also because it gives a uniform description applicable to all types valid for every $R$. 

The automorphism group, Aut$(G(\Phi, R))$, is generated by four of its subgroups. 
We describe these subgroups below and then state the main result. 

Let $\rho:R \to R$ be a ring automorphism. 
This ring automorphism induces, in a natural way, an automorphism of the Chevalley group $G(\Phi, R)$.
Such automorphisms are called \emph{ring automorphisms} of $G(\Phi, R)$. 
If $G(\Phi, R)$ is seen as a matrix group, for example, if $G(\Phi, R) = PSL_n(R)$, then this automorphism is obtained simply by applying the map $\rho$ to the matrix entries. 
We will abuse the notations and use the symbol $\rho$ to denote the ring automorphism of $G(\Phi, R)$ induced by the automorpshim $\rho$ of $R$. 

For every $x \in G(\Phi, R)$, we have the \emph{inner} automorphism $i_x$ given by $i_x(g) = xgx^{-1}$.

An automorphism $\tau:G(\Phi, R) \to G(\Phi, R)$ is called a \emph{central automorphism} if $g^{-1}\tau(g) \in Z(G(\Phi, R))$ for every $g$.
If $\chi: G(\Phi, R) \to Z(G(\Phi, R))$ denotes the map $g \mapsto g^{-1}\tau(g)$ then it follows that $\chi$ is a group homomorphism. 
We then denote the central automorphism $\tau$ by $\tau_{\chi}$. 

We choose a base in the root system $\Phi$ and consider the corresponding Dynkin diagram. This diagram is a graph (with possibly multiple edges) and a graph automorphism of it induces, in a natural way, an automorphism of $G(\Phi, R)$. These automorphisms are called \emph{graph automorphisms} of $G(\Phi, R)$. Note that $G(\Phi, R)$ admits non-trivial graph automorphisms if and only if $G$ is of type $A_l$, $D_l$ or $B_2$ in characteristic $2$.

We now state the main theorem of this section. 

\begin{theorem}\cite[Theorem 1]{Bu}\label{Bunina}
Let $R$ be a commutative ring and let $\Phi$ be an irreducible root system of classical type of rank $> 1$. 
If $\Phi$ is of the type $A_2, B_l$ or $C_l$ then we assume that $1/2 \in R$. 

Every group automorphism of $G(\Phi, R)$ is a composition of a ring automorphism, an inner automorphism, a central automorphism and a graph automorphism. 
\end{theorem}

The commuting relations satisfied by the automorphisms of the above four types are listed, without proof, in the following lemma. 

\begin{lemma}
Let $G = G(\Phi, R)$ denote the Chevalley group of adjoint type associated to the root system $\Phi$ over the ring $R$. 

Let $i_x$ denote the inner automorphism of $G$ corresponding to $x \in G$ and let $\tau_{\chi}$ denote a central automorphism of $G$.
Then for every automorphism $\sigma$ of $G$
$$\sigma \circ i_x = i_{\sigma(x)} \circ \sigma \hskip5mm {\rm and} \hskip5mm \tau_{\chi} \circ \sigma = \sigma \circ \tau_{\eta}$$
where $\eta = \sigma^{-1}|_{Z(G)} \circ \chi \circ \sigma: G \to Z(G)$.

Further, every graph automorphism $\epsilon$ of $G$ commutes with every ring automorphism $\rho$ of $G$. 
\end{lemma}

\begin{corollary}\label{chev-outer}
We assume that the notations are as in the previous lemma. 
Then, every element of $\Aut(G)/{\rm Inn}(G)$ is represented by an automorphism $\sigma \in \Aut(G)$ of the type
$$\sigma = \tau_{\chi} \circ \rho \circ \epsilon .$$
\end{corollary}

\subsection{Ring automorphisms}\label{ring-auto}
We shall assume that $R$ is a ring such that $F[t] \subset R \subsetneq F(t)$, where $F$ is a subfield of $\overline{\mathbb F}_p$. 
An automorphism $\rho:R \to R$ restricts to an automorphism of $F$. 
Note that for any $x\in F$, $\rho(x)=x^{p^r} $ for some $r\ge 1$ where the value of $r$ may depend on $x$. 
Also $\rho(t)=\frac{at+b}{ct+d}$ where $ad-bc \neq 0$ and $a,b,c,d\in \mathbb F_q\subset F$ with $q=p^e$ for some $e\ge 1$.

Note that when $R$ is a ring with $\mathbb{F}_q[t] \subset R \subsetneq \mathbb{F}_q(t)$, then $\Aut(R)$ is a finite group because it is a subgroup of $\Aut(\mathbb{F}_q(t)) = PGL(2,\mathbb{F}_q)\rtimes \Aut(\mathbb{F}_q)$.
If $R_k$ denotes the intersection $\mathbb{F}_k(t) \cap R$, then $\rho$ stabilizes $R_q$ (where $\rho(t)=\frac{at+b}{ct+d}$ where $ad-bc \neq 0$ and $a,b,c,d\in \mathbb F_q\subset F$) and also $R_l$ whenever $R_q \subset R_l$ and $\rho$ is of finite order on each such $R_l$. Consequently, the groups $G(R_l)$ are stable under $\rho$, where $G$ is a classical Chevalley group of adjoint type and $G(R)$ is union of the groups $G(R_l)$. Hence, $\rho$ satisfies the hypothesis of Lemma \ref{fixedgroup}.

Now choose $q$ such that there is a non-zero $f \in \mathbb{F}_q[t]$, which is not invertible in $R$ (this is possible because $R \subsetneq F(t)$) and $\rho$ restricts to an automorphism of $R_q$.
Set $s:=\prod _{\sigma\in \Aut (R_q)} \sigma( f) .$
This is a finite product as $\Aut (R_q)$ is finite. 
Then $s \in R_q$ but is not invertible. 
As, otherwise, $f(t)$, which divides $s$ in $R$, would be invertible in $R$. 
In particular, $s$ is not a scalar. 
Thus we have the following lemma:

\begin{lemma}\label{fixed-set}
Using notations from the above paragraph, define $S$ to be $\mathbb F_p[s]\subset R$.
Then $G(\Phi, S)$ is element-wise fixed under the automorphism $\rho$ of the Chevalley group $G(\Phi, R)$.
\end{lemma}

\subsection{Central automorphisms}

The following lemma shows that there is no non-trivial central automorphisms of the groups of our concern in this paper.

\begin{lemma}[Abe-Hurley, \cite{AH}]\label{Abe-Hurley:1988}
If $G$ is a Chevalley group of adjoint type of rank at least 2, then for any commutative ring $R$ the group $G(R)$ has trivial center .
\end{lemma}

\begin{corollary}\label{adjoint:central}
If $G$ is a Chevalley group of adjoint type of rank at least 2, then for any commutative ring $R$ there are no non-trivial central automorphisms of $G(R)$. 
\end{corollary}

\section{Groups of type $A_n$, $n \geq 2$} \label{A_n}

Throughout this section $R$ denotes a ring such that $F[t] \subset R \subsetneq F(t)$, where $F$ is a subfield of $\bar{\mathbb{F}_p}$, unless stated otherwise.
This section recalls the work done in \cite{MP}, proving that $SL_n(R)$ has the $R_{\infty}$-property for all $n \geq 3$.

The adjoint group isogenous to $SL_n$ is $PSL_n$. 
We prove that $PSL_n(R)$ has the $R_{\infty}$-property and then by Lemma \ref{characteristic}, conclude that the group $SL_n(R)$ also has the $R_{\infty}$-property as every automorphism of $SL_n$ induces an automorphism of $PSL_n$. 
We know from Theorem \ref{Bunina} that a group automorphism of $PSL_n(R)$, $n > 2$, is a product of a ring automorphism, a central automorphism, a graph automorphism and an inner automorphism.

By Corollary \ref{adjoint:central}, there are no non-trivial central automorphisms of $PSL_n(R)$.
Further, there is only one non-trivial graph automorphism, $\epsilon$, of $PSL_n(R)$. 
This automorphism, up to an inner automorphism, is given by $\epsilon(g) = g^{-t}$, the transpose inverse of $g$. 
We use the symbol $\rho$ to denote a ring automorphism of $PSL_n(R)$.
To summarise:

\begin{lemma} \label{out}
Each group automorphism of the group $PSL_n(R)$ is of the type $i_x \circ \rho$ or $i_x \circ \rho \circ \epsilon$. 
\end{lemma}

We note that if $x,y \in PSL_n(R)$ are conjugates, then $tr(x)=\lambda tr(y)$ for some $\lambda \in F^\times$.
This fact, together with the Lemma stated below from \cite{MP} (the proof is included for the sake of completion) proves that the group $PSL_n(R)$ has infinitely many conjugacy classes. It is also crucial in proving the $R_\infty$-property of $PSL_n(R)$.

\begin{lemma}\label{SLn-inf-conju}\cite[Lemma 3.8]{MP} 
Fix an element $s$ of $R$ as in Lemma \ref{fixed-set}. 
Consider the matrix $x_m = \big(\begin{smallmatrix} 1 - s^{2m} & s^m \\ -s^m & 1 \end{smallmatrix}\big)$ identified with $\big(\begin{smallmatrix} x_m & \\ & I_{n-2}\end{smallmatrix}\big) \in SL_n(R)$. 
We identify $x_m$ with its image in $PSL_n(R)$.
For a fixed $r \geq 1$, $deg(tr(x_m^r)) \neq  deg(tr(x_n^r))$, whenever $m \neq n$. 
\end{lemma}

\begin{proof}
Set $x:=\bigl(\begin{smallmatrix} 1-u^2 & u \\ -u& 1\end{smallmatrix} \bigr)\in \SL_2(\mathbb F_p[u])$.  
Then $\trace(x)=2-u^2$ and $\trace(x^2)=2-4u^2+u^4$.
As the characteristic polynomial of $x$ is $X^2-(2-u^2)X+1$, we obtain the relation $\trace(x^r)=(2-u^2)\trace(x^{r-1})-\trace(x^{r-2})$ for any $r\ge 3$.
It follows by induction that $\trace(x^r)$ is a polynomial in $u$ of degree $2r$ with leading coefficient $(-1)^r\in \mathbb F_p$.

The last assertion still holds when $x$ is viewed as an element of $\SL_n(\mathbb F_p[u]), n\ge 3$. 
Applying this to the elements $x_m\in \SL_n(R)$ defined above, $\trace(x_m^r)\in S=\mathbb F_p[s]\subset R$ is a polynomial in $s$ of degree $2rm$.
Hence 
$deg(tr(x_m^r)) \neq  deg(tr(x_n^r))$, whenever $m \neq n$. 
\end{proof}
By Lemma \ref{outer}, the $R_{\infty}$-property for $PSL_n(R)$ will follow if $R(\sigma) = \infty$ for every $\sigma \in T$ where $T \subseteq \Aut(PSL_n(R))$ is a transversal for the quotient $\Aut(PSL_n(R)) \to \Out(PSL_n(R))$. 
Using the lemma \ref{out}, we take $T$ to be the set of automorphisms of $PSL_n(R)$ of the type $\rho$ or $\rho \circ \epsilon$. 

The following theorem is a special case of Theorem 1.1 of \cite{MP}.
We state it below and give the proof for completeness from ibid. $\S$ 3.5.

\begin{theorem}\label{An-finite}
Let $R$ be a ring with $\mathbb{F}_q[t] \subset R \subsetneq \mathbb{F}_q(t)$ and let $G$ be the group $PSL_n(R)$ for $n > 2$. 
Then $G$ has $R_{\infty}$-property. 
\end{theorem}

\begin{proof}
We deal with the case of the automorphism $\rho$ first. 
Every such $\rho$ is of finite order as $\Aut(R)$ is finite as mentioned in \ref{ring-auto}. 
We take $r$ to be the order of $\rho$. 
Since $x_m$, in Lemma  \ref{SLn-inf-conju}, are fixed by $\rho$, Lemma \ref{fixedgroup} and Lemma \ref{SLn-inf-conju} give that $R(\rho) = \infty$. 

Now we come to the case of $\rho \circ \epsilon$.
We first discuss the subcase of $\rho$ being trivial, that is the case of the automorphism $\epsilon$. 
We identify the matrix $e_{12}(a) = \big(\begin{smallmatrix} 1 & a \\ & 1\end{smallmatrix}\big)$ with the element $\big(\begin{smallmatrix} e_{12}(a) & \\ & I_{n-2}\end{smallmatrix}\big) \in SL_n(R)$ and with its image in $PSL_n(R)$.
Lemma \ref{fixedgroup} along with the fact that $x_m=e_{12}(s^m)\epsilon(e_{12}(s^m))$ implies that the identified elements $e_{12}(s^m)$ and $e_{12}(s^k)$ in $PSL_n(R)$, $m \ne k$, are in different $\epsilon$-twisted conjugacy classes. 

A similar argument is applicable in the general case of an automorphism $\rho \circ \epsilon$. 
We again see that the elements $e_{12}(s^m)$ and $e_{12}(s^k)$, $m \ne k$, of $PSL_n(R)$ are in different $\rho \circ \epsilon$-twisted conjugacy classes.
\end{proof}

Now, we state and give the proof of ibid. Theorem 1.1 in the general case where $F$ is a subfield of $\overline{\mathbb{F}}_q$. 

\begin{theorem}
Let $F$ be a subfield of $\overline{\mathbb{F}}_p$ and take $R$ to be a ring such that $F[t] \subset R \subsetneq F(t)$. 
The group $PSL_n(R)$, for $n \geq 3$, has the $R_{\infty}$-property. 
\end{theorem}

\begin{proof}
Without loss of any generality, we take $F$ to be infinite. 
Once again, it is enough to consider the cases of the automorphisms $\rho$ and $\rho \circ \epsilon$. 

Let $\rho \in \Aut(R)$ be defined over $\mathbb{F}_q$ for some $q$. If $\mathbb{F}_q \subset \mathbb{F}_l$, then $\rho$ restricts to an automorphism of $R_l=R \cap \mathbb{F}_l(t)$ as discussed in section \ref{ring-auto}.

Let $x_m=e_{12}(s^m)e_{21}(-s^m)\in G_q=PSL_n(R_q), m\ge 1,$.
Then $x_m\in \fix(\rho)$.  
Suppose that there exists an element $z\in G = PSL_n(R)$ such that $x_k=z x_m\rho(z^{-1})$ with $k\ne m$.
There exists $\ell=q^d=p^{de}$ a sufficiently large power of $q$ such that $\mathbb F_\ell\subset F$ and  $x_k,x_m, z\in G_\ell$.
Then $\rho^N|_{G_\ell}=id$ where $N:=de$.  
This implies, by Lemma \ref{fixedgroup} that $x^N_m$ and $x^N_k$ are conjugates in $G_\ell$. 
This contradicts Lemma \ref{SLn-inf-conju} and we conclude that $R(\rho)=\infty$.
The proof for $\rho\circ \epsilon$ is similar to the proof of the corresponding type give above, Theorem \ref{An-finite}.
\end{proof}

\section{Groups of type $C_n$, $n \geq 2$} \label{C_n}

The adjoint group of type $C_n$ is $PSp_{2n}$. 
In this section, we will prove that $PSp_{2n}(R)$ has the $R_\infty$-property, thereby conclude that so has $Sp_{2n}(R)$ by Lemma \ref{characteristic}.

By Corollary \ref{adjoint:central}, there are no non-trivial central automorphisms of $PSp_{2n}(R)$.
Further, $PSp_{2n}$ has no non-trivial graph automorphism. So by Theorem \ref{Bunina}, using notations from \S 3, we have the following lemma.

\begin{lemma}
Each group automorphism of the group $PSp_{2n}(R)$ is of the form $i_x \circ \rho$. 
\end{lemma}

Set $x_m=e_{12}(s^m)e_{21}(-s^m)\in \SL_2(S)$, where $S$ is an in Lemma \ref{fixed-set} so that $x_m= \bigl(\begin{smallmatrix} 1-s^{2m} & s^m \\ -s^m & 1\end{smallmatrix} \bigr)$. 
Then $x_m^{-1}=e_{21}(s^m)e_{12}(-s^m)= \bigl(\begin{smallmatrix} 1 & -s^m \\ s^m & 1-s^{2m}\end{smallmatrix} \bigr)$. 
Both $x_m$ and $x_m^{-1}$ satisfy the polynomial $X^2+(s^{2m}-2)X+I_2=0$.   
We regard $x_m$ and $x_m^{-1}$ also as elements of $\SL_n(S)$ by identifying them with the block diagonal matrices $\delta(x_m,I_{n-2})$ and $\delta(x_m^{-1},I_{n-2})$ respectively. 
Note that $y_m = \bigl(\begin{smallmatrix} x_m &  \\  & x_m^{-1}\end{smallmatrix} \bigr) \in Sp_{2n}$.
The following lemma, which is similar to Lemma \ref{SLn-inf-conju} will play a crucial role in our proof.

\begin{lemma}\label{trace}
Let $F$ be a subfield of the algebraic closure of a finite field $\mathbb{F}_p$, $p \ne 2$, and $R$ be a ring such that $F[t] \subset R \subsetneq F(t)$. 
Fix $r\ge 1$. 
Let $y_m=\bigl(\begin{smallmatrix} x_m &  \\  & x_m^{-1}\end{smallmatrix} \bigr) \in Sp_{2n}, m > 1,$ where $x_m=e_{12}(s^m)e_{21}(-s^m)$ and $s$ is as in Lemma $\ref{fixed-set}$. 
Then $deg(tr(y_m^r)) \neq  deg(tr(y_l^r))$, whenever $l \neq n$. 
\end{lemma}

\begin{proof}
First we note that $y_m^r=\bigl(\begin{smallmatrix} x_m^r &  \\  & x_m^{-r}\end{smallmatrix} \bigr)$ and $\trace(x_m^r)=\trace(x_m^{-r})$. 
So $\trace(y_m^r)=2\trace(x_m^r).$ Since $p \neq 2$, it follows from Lemma \ref{SLn-inf-conju} that  $deg(tr(y_m^r)) \neq  deg(tr(y_l^r))$, whenever $l \neq n$.  
\end{proof}

Let $\sigma \in \Aut(PSp_{2n}(R))$, where $\sigma$ is of the form $\sigma = \rho$ as mentioned in corollary \ref{chev-outer}. 
Note that in order to show that $PSp_{2n}(R)$ has the $R_\infty$-property, by Lemma \ref{outer}, it is enough to show that $R(\sigma)=\infty.$
Suppose $x,y \in \fix(\rho)$ such that $x,y$ are $\sigma$-conjugates.
So $x=zy \sigma(z^{-1})$, for some $z \in PSp_{2n}(R)$.
Since $\rho$ satisfies the hypothesis of Lemma \ref{fixedgroup}, we conclude that there exists a natural number $r$ such that $x^r$ is conjugate to $y^r$ in $PSp_{2n}(R)$. Since the center of $Sp_{2n}(R)$ is $\{\pm I_n\}$, $x^r$ is conjugate to $y^r$ in $PSp_{2n}(R)$ implies $tr(x^r)=\pm tr(y^r)$, i.e. $deg(tr(x^r))=deg(tr(y^r)).$

Thus we conclude that $\{y_m^r : m>1 \}$ is an infinite set of elements of $PSp_{2n}(R)$ that are in pairwise disjoint conjugacy classes in $PSp_{2n}(R)$.
Hence we have the following theorem.

\begin{theorem}
Let $F$ be a subfield of the algebraic closure of a finite field $\mathbb{F}_p$, $p \ne 2$, $R$ be a ring such that $F[t] \subset R \subsetneq F(t)$ and $PSp_{2n}(R), n \geq 2$.
Then $PSp_{2n}(R)$ and so $Sp_{2n}(R)$ have the $R_\infty$-property.
\end{theorem}

\section{Groups of type $B_n$, $n \geq 2$} \label{B_n}
The adjoint group of type $B_n$ over $R$ is $G=PSO_{2n+1}(R)$, which is equal to the group $SO_{2n+1}(R)=\{g \in SL_{2n+1}(R) \mid {}^{t}gAg=A \}$, where 
$$A= \left[ 
\begin{array} {ccc}
  0_n & I_n & 0 \\  
  I_n & 0_n & \vdots \\
  0 & \dots & 1
\end{array} 
\right].$$
By Corollary \ref{adjoint:central}, $G = PSO_{2n+1}(R)$ admits no non-trivial central automorphisms.
It also does not have any non-trivial graph automorphism. 
So by Theorem \ref{Bunina}, using notations from $\S$ 3, we have the following lemma.

\begin{lemma}
Each group automorphism of the group $G=SO_{2n+1}(R)=PSO_{2n+1}(R)$ is of the form $i_x \circ \rho$. 
\end{lemma}

Let $\sigma \in \Aut(G)$, where $\sigma$ is of the form $\sigma = \rho$ as mentioned in corollary \ref{chev-outer}. Suppose $x,y \in \fix(\rho) \subset G$ such that $x,y$ are $\sigma$-conjugates. Since $\rho$ satisfies the hypothesis of Lemma \ref{fixedgroup}, we conclude that there exists a natural number $r$ such that $x^r$ is conjugate to $y^r$ in $G$.
For $\lambda (\neq 0)\in R$ let 
$$x_\lambda=\left[ 
\begin{array} {ccc}
  I_n & y_\lambda & 0 \\  
  0_n & I_n & \vdots \\
  0 & \dots & 1
\end{array} 
\right] {\rm ~where~} y_\lambda = \left[ 
\begin{array} {ccccc}
  0 & -\lambda & 0 & \dots & 0 \\  
  \lambda & 0 & 0 & \dots & 0 \\
  0 & 0 & 0 & \dots & 0 \\
  \vdots & \vdots & \vdots & \vdots & \vdots \\
  0 & 0 & 0 & \dots & 0
\end{array}
\right].$$

Note that $x_\lambda \in G$.
A direct calculation shows that $x_\lambda^r =
x_{r \lambda}$, for $n \geq 2$.
Now let $x_{r \lambda}, x_{r \lambda'}$ be conjugates in $G$.
Then there exists $g \in G$ such that $x_{r \lambda}g=gx_{r \lambda'}$.
We let
$$g= \left[ 
\begin{array} {ccc}
  K & L & a_1 \\  
  M & N & \vdots \\
  b_1 & \dots & a_{2n+1}
\end{array} 
\right]$$ 
Then we have $y_{r \lambda}M=0_n$ and $My_{r \lambda'}=0_n$, solving which, we see that each entry of the first two rows and the first two columns of $M$ must be $0$.

We further have $a_1 = a_1-\lambda a_{n+2}$, $a_2 = a_2+\lambda a_{n+1}$, $b_{n+1} = -\lambda b_2+b_{n+1}$ and $b_{n+2} = \lambda b_1+b_{n+2}$. 
Solving these equations, we get that $a_{n+2} = a_{n+1} = b_2 = b_1 = 0$.

We now consider the equation: $y_{r \lambda}N=Ky_{r \lambda'}$ with $N=(n_{ij})$ and $K=(k_{ij})$. 
By solving it, we see that $n_{1r}=0$ for all $r >2$, $n_{2r}=0$ for all $r>2$, $k_{r1}=0$ for all $r>2$ and $k_{r2}=0$ for all $r>2$. 
We also get
\begin{eqnarray}
-\lambda n_{21}& = &\lambda' k_{12},\\
\lambda n_{22} & = & \lambda' k_{11},\\
\lambda n_{11} & = & \lambda' k_{22}, \\
\lambda n_{12} & =  & -\lambda' k_{21} .
\end{eqnarray}

If we let $K_1$ and $N_1$ denote the top-left $2 \times 2$ parts of $K$ and $N$, respectively then $det(K_1)det(N_1)$ divides $det(g) = 1$.
Thus, both $\det(K_1)$ and $\det(N_1)$, are units in $R$ and so is their product, $\alpha = det(K_1)det(N_1)$. 
Then 
$$\alpha = (k_{11}k_{22}-k_{21}k_{12})(n_{11}n_{22}-n_{21}n_{12})$$
and hence 
$$\lambda'^2 \alpha = (\lambda' k_{11} \lambda' k_{22}-\lambda' k_{21}\lambda' k_{12})(n_{11}n_{22}-n_{21}n_{12}) .$$
By (6.1)--(6.4), we get 
$$\lambda'^2 \alpha = (\lambda n_{22} \lambda n_{11}- \lambda n_{12} \lambda n_{21})(n_{11}n_{22}-n_{21}n_{12}) = \lambda^2det(N_1)^2 .$$
Therefore, $\lambda'^2=\lambda^2det(N_1)^2 \alpha^{-1}$ or $\lambda'^2=\lambda^2 u$, for a unit $u$ in $R$.
Thus, if $x_\lambda^r,x_{\lambda'}^r$ are conjugates in $G$ then $\lambda'^2=\lambda^2 u$, for a unit $u$ in $R$.

Now take $s \in R$, where $s$ is as in Lemma \ref{fixed-set}. Note that $s$ is not a unit in $R$ and $s \notin F$ as discussed in \S \ref{ring-auto}.
Thus it follows from the above calculation that $\{x_{s^k} \mid k \in \mathbb{N}\} \subset Fix(\rho)$ is an infinite subset of $G$ such that $x_{s^k}^r$ is not conjugate to $x_{s^{k'}}^r$ for $k \neq k'$. 
We then have the following theorem.

\begin{theorem}
Let $F$ be a subfield of the algebraic closure of a finite field $\mathbb{F}_p$, $p \ne 2$, $R$ be a ring such that $F[t] \subset R \subsetneq F(t)$ and $G=PSO_{2n+1}(R)=SO_{2n+1}(R), n \geq 2$. Then $G$ has $R_\infty$-property.
\end{theorem}

\section{Groups of type $D_n$, $n \geq 5$} \label{D_n}

The adjoint group of type $D_n$ is $PSO_{2n}$. 
Here, $SO_{2n}(R)=\{g \in SL_{2n}(R) \mid {}^{t}gAg=A \}$, where  
$$A= \left[ 
\begin{array} {ccc}
  0_n & I_n \\  
  I_n & 0_n 
\end{array} 
\right].$$ 
By Corollary \ref{adjoint:central}, $G = PSO_{2n}(R)$ admits no non-trivial central automorphisms.

The groups of type $D_n$ admit graph automorphisms. 
If $n \ne 4$ then there is exactly one non-trivial automorphism of the Dynkin diagram of $D_n$. 
This graph automorphism gives rise to an automorphism of the group of type $D_n$ by fixing the diagonal maximal torus pointwise and simply permuting the root subgroups appropriately. 
One sees from \cite[$\S$ 14.5]{Ca} that in the usual matrix representation of $SO_{2n}$ in $GL_{2n}$, this automorphism corresponds to conjugation of the entries of $SO_{2n}$ by the matrix
$$B = \begin{bmatrix}
I_n & & & \\
 & 0 & & 1 \\
 & & I_{n-2} & \\
 & 1 & & 0
\end{bmatrix} .$$

So by Theorem \ref{Bunina}, using notations from \S 3, we have the following lemma.

\begin{lemma}
Each group automorphism of the group $G=PSO_{2n}(R)$ is of the form $\sigma=i_x \circ i_B \circ \rho$, $n \geq 3, n \neq 4$. 
\end{lemma}

Let $\sigma \in \Aut(G)$, where $\sigma$ is of the form $\sigma = i_B \circ\rho$ as mentioned in corollary \ref{chev-outer}.
Suppose $x,y \in \fix(\rho) \subset G$ such that $x,y$ are $\sigma$-conjugates. 
Since $\rho$ satisfies the hypothesis of the Lemma \ref{fixedgroup} and $\rho(B)=B$, we conclude that there exists an even natural number $r$ such that $(xB)^r$ is conjugate to $(yB)^r$ in $G$.

For $\lambda \in R,$ let 
$$x_\lambda= 
\begin{bmatrix}
I_n & y_\lambda \\  
0_n & I_n 
\end{bmatrix} \hskip5mm {\rm where} \hskip5mm
y_\lambda = \left[ 
\begin{array} {ccccc}
  0 & -\lambda & 0 & \dots & 0 \\  
  \lambda & 0 & 0 & \dots & 0 \\
  0 & 0 & 0 & \dots & 0 \\
  \vdots & \vdots & \vdots & \vdots & \vdots \\
  0 & 0 & 0 & \dots & 0
\end{array}
\right].$$

Note that $x_\lambda \in SO_{2n}(R)$. 
A direct calculation shows that $(x_\lambda B)^r 
=x_{r \lambda}$, for $n \geq 3$ and for any even integer $r$.
We denote the image of $x_\lambda$ in $G=PSO_{2n}(R)$ by the same symbol $x_\lambda$.
Now let $x_{r \lambda}$ and $x_{r \lambda'}$ be conjugates in $G = PSO_{2n}(R)$. 
Then there exists $g \in SO_{2n}(R)$ such that $g^{-1}x_{r \lambda}g=\pm x_{r \lambda'} $ holds in $SO_{2n}$.
This gives $g^{-1}x_{r \lambda}^2g=x_{r \lambda'}^2$, i.e. $x_{2r \lambda}g=gx_{2r \lambda'}$. This is because $x_\lambda^r =x_{r \lambda}$, for $n \geq 2$ as noted in the previous section.
If we let $2r=p$ and
$$g=\begin{bmatrix}
  K & L \\ 
  M & N
\end{bmatrix}$$
then we have $y_{p \lambda}M=0_n$ and $My_{p \lambda'}=0_n$, solving which, we see that each entry of the first two rows and the first two columns of $M$ must be $0$.

We now consider the equation: $y_{p \lambda}N=Ky_{p \lambda'}$ with $N=(n_{ij})$ and $K=(k_{ij})$. 
By solving it, we see that $n_{1r}=0$ for all $r >2$, $n_{2r}=0$ for all $r>2$, $k_{r1}=0$ for all $r>2$ and $k_{r2}=0$ for all $r>2$. 
We also get
\begin{eqnarray}
-\lambda n_{21} & = & \lambda' k_{12}, \\
\lambda n_{22} & = & \lambda' k_{11},\\
\lambda n_{11} & = & \lambda' k_{22},\\
\lambda n_{12}& = & -\lambda' k_{21}.
\end{eqnarray}

If we let $K_1$ and $N_1$ denote the top-left $2 \times 2$ parts of $K$ and $N$, respectively then $det(K_1)det(N_1)$ divides $det(g) = 1$.
Thus, both $\det(K_1)$ and $\det(N_1)$, are units in $R$ and so is their product, $\alpha = det(K_1)det(N_1)$. 
Then 
$$\alpha = (k_{11}k_{22}-k_{21}k_{12})(n_{11}n_{22}-n_{21}n_{12})$$
and hence 
$$\lambda'^2 \alpha = (\lambda' k_{11} \lambda' k_{22}-\lambda' k_{21}\lambda' k_{12})(n_{11}n_{22}-n_{21}n_{12}) .$$
By (7.1)--(7.4), we get 
$$\lambda'^2 \alpha = (\lambda n_{22} \lambda n_{11}- \lambda n_{12} \lambda n_{21})(n_{11}n_{22}-n_{21}n_{12}) = \lambda^2det(N_1)^2 .$$
Therefore, $\lambda'^2=\lambda^2det(N_1)^2 \alpha^{-1}$ or $\lambda'^2=\lambda^2 u$, for a unit $u$ in $R$.
Thus, if $x_{p\lambda},x_{p\lambda'}$ are conjugates in $G$ then $\lambda'^2=\lambda^2 u$, for a unit $u$ in $R$.
Thus if $(x_\lambda B)^r,(x_{\lambda'} B)^r$ are conjugates in $G$ then $\lambda'^2=\lambda^2 u$, for some $u \in R^\times$.

Now take $s \in R$, where $s$ is as in Lemma \ref{fixed-set}. Note that $s$ is not a unit in $R$ and $s \notin F$ as discussed in \S \ref{ring-auto}.
Thus it follows that $\{x_{s^k} \mid k \in \mathbb{N}\}$ is an infinite subset of $G$ such that $(x_{s^k}B)^r$ is not conjugate to $(x_{s^{k'}}B)^r$ whenever $k \neq k'$.
Hence we have the following theorem.

\begin{theorem}
Let $F$ be a subfield of the algebraic closure of a finite field $\mathbb{F}_p$, $p \ne 2$, $R$ be a ring such that $F[t] \subset R \subsetneq F(t)$, $n \geq 3, n \neq 4$.
Then $PSO_{2n}(R)$ and hence $SO_{2n}(R)$ have the $R_\infty$-property.
\end{theorem}


\section{Groups of type $D_4$} \label{D_4}

In this section, we prove that $D_4$ satisfies the $R_{\infty}$-property over a ring $R$ such that $F[t] \subset R \subsetneq F(t)$.
The reason we are treating $D_4$ separately is the existence of the automorphism of order three of its Dynkin diagram. 
The Dynkin diagram of $D_4$ affords the group $\mathcal{S}_3$ as its symmetry group, unlike the Dynkin diagrams of other $D_n$ whose graph automorphism groups are all isomorphic to $\mathbb{Z}/2$. 
These elements of $\mathcal{S}_3$ give rise to outer automorphisms of the Lie algebra of type $D_4$ in a natural way.

Let us write this group $\mathcal{S}_3$ as $\{1, \tau_1, \tau_2, \tau_3, \sigma, \sigma^2\}$.
Here $\tau_i$ are reflections and $\sigma^i$ are rotations, acting on the Dynkin diagram of $D_4$.
These automorphisms of the Lie algebra give outer automorphisms of the simply connected group of type $D_4$, the group $Spin(8)$, and also the adjoint group of type $D_4$, the group $PSO(8)$.
However, the group $SO(8)$ does not admit all these outer automorphisms. 
It can be seen as follows. 

The center of $Spin(8)$ is the Klein four group $\mathbb{Z}/2 \times \mathbb{Z}/2$ and the $\mathcal{S}_3$ above permutes the non-trivial elements of the center. 
The three quotients of $Spin(8)$ by each of the order $2$ central subgroups are each isomorphic to $SO(8)$.
These three quotients are permuted by the rotations, $\sigma$ and $\sigma^2$, while each of those three quotients admit one outer automorphism, a reflection in $\mathcal{S}_3$, the one which preserves the corresponding subgroup of the center of $Spin(8)$. 

Now, consider the quotient $SO(8)$ of $Spin(8)$ which admits the automorphism $\tau_1$. 
It follows from the previous section that $\tau_1$ on this $SO(8)$ has infinite Reidemeister number. 
Then by Lemma \ref{characteristic}, the reflection $\tau_1$ has infinite Reidemeister number on $Spin(8)$ as well as on $PSO(8)$. 
The other two quotients of $Spin(8)$ by order two central subgroups are the images of this $SO(8)$ under $\sigma$ and $\sigma^2$.
They admit the outer graph automorphism $\sigma \circ \tau_1 \circ \sigma^{-1} = \tau_2$ and $\sigma \circ \tau_2 \circ \sigma^{-1} = \tau_3$. 
It then follows, simply using the isomorphism of groups, that the automorphisms $\tau_2$ and $\tau_3$ are also taken care of. 

We now concentrate on the automorphism $\sigma$, the work for $\sigma^2$ will be similar and hence will not be mentioned here. 

We note that in the previous section, the $y_{\lambda}$ are actually fixed by the diagram automorphism $\tau_1$. 
In the following lemma, we use notations from the previous section.
The proof of this lemma follows on the same lines as the one in the previous section, hence it will not be repeated here. 

\begin{lemma}
Let $G= D_4$ over a ring $R$ such that $F[t] \subset R \subsetneq F(t)$ and consider an automorphism $\phi$ of $G$ of the form $\phi = \rho \circ \epsilon$, where $\rho$ is a ring automorphism and $\epsilon$ is a graph automorphism.
Then $x, y \in Fix(\rho) \cap Fix(\epsilon)$ are $\phi$-conjugates implies that $x^l$ is conjugate to $y^l$ in $G$, for some integer $l$.
\end{lemma} 

Now, we observe that the automorphism $\epsilon$ fixes the central root, say $\alpha$, in the Dynkin diagram of $D_4$. 
Hence the automorphism $\epsilon$ fixes the corresponding root subgroup $U_{\alpha}$ pointwise. 
Since all roots of $D_4$ are of the same length, the Weyl group $W(D_4)$ operates transitively on the root system.

It is proved in the previous section that there is a root subgroup of $D_4$, which has an infinite subset $\{x_{s^k} \mid k \in \mathbb{N}\}$, for an appropriate $s$, such that $(x_{s^k}B)^r$ is not conjugate to $(x_{s^{k'}}B)^r$ whenever $k \neq k'$.
By the Weyl group action, we have that the root subgroup $U_{\alpha}$ also has one such set. 

Together with the proofs for automorphisms that do not involve the graph automorphism, exactly as done in the case $n \ne 4$ in the previous section, this completes the proof of the following result.

\begin{theorem}
Let $F$ be a subfield of the algebraic closure of a finite field $\mathbb{F}_p$, $p \ne 2$, $R$ a ring such that $F[t] \subset R \subsetneq F(t)$ and $G$ be the adjoint group of type $D_4$ over $R$.
Then $G$ has $R_\infty$-property.
\end{theorem}



\begin{thebibliography}{99}
\bibitem{Fel} A. Felshtyn, \emph{``The Reidemeister number of any automorphism of Gromov hyperbolic group is infinite''}, Zapiski Nauchnych Seminarov, POMI, 279 (2001), 229-241.
\bibitem{felshtyn-hill} A. Felshtyn; R. Hill,  \emph{``The Reidemeister zeta function with applications to Nielsen theory and a connection with Reidemeister torsion''}, K-Theory 8 (1994), 367--393.
\bibitem{FN} A. Felshtyn; T. Nasybullov, \emph{``The $R_\infty$ and $S_\infty$ properties for linear algebraic groups''}, J. Group Theory, 19 (2016), 901-921.
\bibitem{gs-pjm-2016} D. L. Gon\c calves; P. Sankaran, \emph{``Sigma theory and twisted conjugacy, II: Houghton groups and pure symmetric automorphism groups''}, Pacific Jour. Math. 280 (2016), 349--369.
\bibitem{g-s} D. L. Gon\c calves; P. Sankaran, \emph{``Twisted Conjugacy in Richard Thompson's Group $T$''}, arXiv:1309.2875v2 [Math:GR].
\bibitem{AH} E. Abe, J. F. Hurley, \emph{``Centers of Chevalley groups over commutative rings"}, Comm. Algebra 16 (1988), 57--74.
\bibitem{Bu} E. I. Bunina, \emph{``Automorphisms of Chevalley groups of different types over commutative rings"}, J. Algebra 355 (2012), 154--170. 
\bibitem{S} F. Grunewald; J. Mennicke; L. Waserstein, \emph{``On symplectic groups over polynomial rings"}, Math. Z. 206 (1991), no. 1, 35--56. 
\bibitem{LL} G. Levitt, M. Lustig, \emph{`Most automorphisms of a hyperbolic group have very simple dynamics''}, Ann. Scient. Ec. Norm. Sup., 33 (2000), 507-517.
\bibitem{Hu} J. E. Humphreys, \emph{``Introduction to Lie algebras and representation theory"}, Second printing, revised, GTM 9, Springer-Verlag, 1978.
\bibitem{mitra-sankaran-n2} Mitra, O; Sankaran, P. Twisted conjugacy in $\GL_2$ and $\SL_2$ over polynomial rings over a finite field. to appear in Geom. Dedicata.
\bibitem{MP} Mitra, O; Sankaran, P, \emph{``Twisted conjugacy in $SL_n$ and $GL_n$ over subrings of $\bar{\mathbb{F}_p}(t)$"}, arXiv:1912.10184 [Math:GR].
\bibitem{Y} P. Lins.; Y. Santos. \emph{``Twisted conjugacy in soluble arithmetic groups''}, arXiv:2007.02988 [math.GR].
\bibitem{Ca} Roger W. Carter \emph{``Simple groups of Lie type"}, Reprint of the 1972 original. Wiley classics library. 1989.
\bibitem{steinberg} R. Steinberg, \emph{``Endomorphisms of Linear Algebraic Groups''}, Memoirs of AMS, {\bf 80} 1968. 
\bibitem{St} R. Steinberg \emph{``Lectures on Chevalley groups"}, Notes prepared by John Faulkner and Robert Wilson. Yale University, New Haven, Conn., 1968. 
\bibitem{ms-cmb} T. Mubeena; P. Sankaran,  \emph{`Twisted conjugacy classes in abelian extensions of certain linear groups''}, Canad. Math. Bull. {\bf  57} (2014), 132--140.
\bibitem{ms-tg} T. Mubeena; P. Sankaran, \emph{`Twisted conjugacy classes in lattices in semisimple Lie groups''}, Transform. Groups {\bf 19} (2014), 159--169. 
\bibitem{T} T. Nasybullov, \emph{``Twisted conjugacy classes in general and special linear groups"}, Algebra Logika 51.3 (2012), 331-346. 
\bibitem{N} T. Nasybullov, \emph{``The $R_\infty$-property for Chevalley groups of types $B_l,C_l, D_l$ over integral domains''}, J. Algebra, {\bf 446} (2016),  489--498. 
\bibitem{T1} T. Nasybullov, \emph{``Twisted conjugacy classes in unitriangular groups"}, J. Group Theory 22.2 (2019), 253-266. 

\end{thebibliography}
\end{document}